\documentclass{amsart}

\usepackage{macros}
\standardsettings
\colorcommentstrue

\usepackage[top=1in, left=1in, right=1in, bottom=1in]{geometry}
\usepackage{epsfig}
\DeclareMathAlphabet{\mathpzc}{OT1}{pzc}{m}{it}
\newcommand{\df}{{\, \stackrel{\mathrm{def}}{=}\, }}

\draftfalse

\begin{document}

\title[Determinacy of metric games]{Determinacy and indeterminacy of games played on complete metric spaces}

\authorlior\authortue\authordavid

\maketitle

\begin{abstract}
Schmidt's game is a powerful tool for studying properties of certain sets which arise in Diophantine approximation theory, number theory, and dynamics. Recently, many new results have been proven using this game. In this paper we address determinacy and indeterminacy questions regarding Schmidt's game and its variations, as well as more general games played on complete metric spaces (e.g. fractals). We show that except for certain exceptional cases, these games are undetermined on Bernstein sets.
\end{abstract}

\maketitle
\section{Introduction}

In 1966, W. M. Schmidt \cite{Schmidt1} introduced a two-player game referred to thereafter as Schmidt's game. Schmidt invented the game primarily as a tool for studying certain sets which arise in number theory and Diophantine approximation theory. These sets are often exceptional with respect to both measure and category. The most significant example is the following. Let $\Q$ denote the set of rational numbers. A real number $x$ is said to be \emph{badly approximable} if there exists a positive constant $c = c(x)$ such that $\left|x-\frac{p}{q}\right| > \frac{c}{q^2}$ for all $\frac{p}{q}\in \Q$. We denote the set of badly approximable numbers by $\BA$. This set plays a major role in Diophantine approximation theory, and is well known to be both meager and Lebesgue null. Nonetheless, using his game, Schmidt was able to prove the following remarkable result:

\begin{theorem}[Schmidt \cite{Schmidt1}]
Let $(f_n)_{n=1}^{\infty}$  be a sequence of $\CC^1$ diffeomorphisms of $\R$. Then the Hausdorff dimension of the set $\bigcap_{n=1}^{\infty}f^{-1}_n(\BA)$ is $1$. In particular, $\bigcap_{n=1}^{\infty}f^{-1}_n(\BA)$ is uncountable.
\end{theorem}

Before turning our attention to the determinacy aspects of the game, we devote the first subsection of the introduction to describing the game as well as recent variations introduced by C. T. McMullen \cite{McMullen_absolute_winning}, both of which will be the focus of this paper. In the second subsection of the introduction we state the main results of this paper.

\subsection{Schmidt's game and its variations}
\label{subsectiongames}

Let $(X,\dist)$ be a complete metric space. In what follows, we denote by $B(x,\rho)$ the closed ball in the metric space $(X,\dist)$ centered at $x$ of radius $\rho$, i.e.,
\begin{equation}
\label{ballsdef}
B(x,\rho) \df \{ y \in X : d(x,y) \le \rho\}\,.
\end{equation}
Let $\Omega \df X \times \R_+$ be the set of formal balls in $X$, and define a partial ordering on $\Omega$ by letting
\[
(x_2,\rho_2)\le_{s}(x_1,\rho_1) \text{ if } \rho_2+d(x_1,x_2)\le \rho_1.
\]
We associate to each pair $(x,\rho)$ a ball in $(X,\dist)$ via the `ball' function $B(\cdot ,\cdot)$ as in (\ref{ballsdef}).
%%y
%\begin{center}
%$B_X(x,\rho)\df \{ y \in X : d(x,y) \le \rho\}$.
%\end{center}
% We will simply write $B(x,\rho)$ when the choice of the metric space is clear
%from the context.
Note that the inequality $(x_2,\rho_2)\le_{s}(x_1,\rho_1)$ clearly implies (but is not necessarily implied by) the inclusion $B(x_2,\rho_2) \subset B(x_1,\rho_1)$. Nevertheless, the two conditions are equivalent when $(X,\dist)$ is a Banach space.

Fix $\alpha,\beta\in (0,1)$ and $S\subset X$. The set $S$ will be called the \emph{target set}. Schmidt's $(\alpha,\beta,S)$-game is played by two players, whom we shall call Alice and Bob. The game starts with Bob choosing a pair $\omega_1 = (x_1,\rho_1) \in \Omega$. Alice and Bob then take turns choosing pairs $\omega'_n = (x'_n,\rho'_n)\le_s\omega_n$ and $\omega_{n+1}= (x_{n+1},\rho_{n+1})\le_s\omega'_n$, respectively. These pairs are required to satisfy
\begin{equation}
\label{Schmidt_rules}
\rho_n' = \alpha \rho_n\text{ and }\rho_{n+1} = \beta \rho_n'\,.
\end{equation}
Since the game is played on a complete metric space and since the diameters of the nested  balls
\begin{equation}
\label{nestedballs}
B(\omega_1) \supset  \ldots\supset B(\omega_n) \supset B(\omega'_n) \supset B(\omega_{n + 1}) \supset \ldots
\end{equation}
tend to zero as $n\rightarrow\infty$, the intersection of these balls is a singleton $\{x_\infty\}$. Call Alice the winner if $x_\infty\in S$; otherwise Bob is declared the winner. A \emph{strategy} consists of a description of how one of the players should act based on the opponent's previous moves; see Section \ref{sectiongeneral} for a precise formulation. A strategy is \emph{winning} if it guarantees the player a win regardless of the opponent's moves. If Alice has a winning strategy for Schmidt's $(\alpha,\beta,S)$-game, we say that $S$ is an $(\alpha,\beta)$-\emph{winning} set. If $S$ is $(\alpha,\beta)$-winning for all (equiv. for all sufficiently small) $\beta\in (0,1)$, we say that $S$ is an \emph{$\alpha$-winning} set. If $S$ is $\alpha$-winning for some (equiv. for arbitrarily small) $\alpha\in (0,1)$, we say that $S$ is \emph{winning}.\footnote{To see that ``for all'' and ``for some'' may be replaced by ``for all sufficiently small'' and ``for arbitrarily small'' respectively, cf. \cite[Lemmas 8 and 9]{Schmidt1}.}

We now describe two variations of Schmidt's game introduced by C. T. McMullen \cite{McMullen_absolute_winning}. The first of these variations is the \emph{strong winning} game. Given $\alpha,\beta\in (0,1)$ and $S\subset X$, the rules of the $(\alpha,\beta,S)$-strong winning game are the same as the rules of Schmidt's $(\alpha,\beta,S)$-game, except that the equations \eqref{Schmidt_rules} are replaced by the inequalities
\begin{equation}
\label{strong_winning_rules}
\rho_n' \geq \alpha \rho_n\text{ and }\rho_{n+1} \geq \beta \rho_n'.
\end{equation}
However, since \eqref{strong_winning_rules} is insufficient to ensure that the diameters of the nested balls \eqref{nestedballs} tend to zero, it may happen that the intersection $I \df \bigcap_n B(\omega_n)$ is not a singleton. If this occurs, we call Alice the winner if $I\cap S\neq\emptyset$; otherwise Bob is declared the winner. If Alice has a winning strategy for the $(\alpha,\beta)$-strong winning game with a given target set $S$, then $S$ is called \emph{$(\alpha,\beta)$-strong winning}. Similarly, as above one can talk about $\alpha$-\emph{strong winning} and \emph{strong winning} sets.

McMullen's second variation is called the \emph{absolute winning} game. Although he defined it only on $\R^d$, it can be played on much more general spaces. Fix $c > 0$. A metric space $(X,\dist)$ is said to be \emph{$c$-uniformly perfect} if for every pair $(x,\rho)\in\Omega$ with $\rho\leq 1$, we have
\begin{equation}
\label{cuniformlyperfect}
B(x,\rho)\butnot B(x,c \rho)\neq\emptyset.
\end{equation}
If $(X,\dist)$ is $c$-uniformly perfect for some $c > 0$, then we say that $(X,\dist)$ is \emph{uniformly perfect}. Examples of uniformly perfect spaces include $\R^d$ as well as many fractal subsets including the Cantor ternary set, the Sierpinski triangle, and the von Koch snowflake curve.

Let $(X,\dist)$ be a complete $c$-uniformly perfect metric space, and fix $0 < \beta < c/5$ and $S\subset X$. The \emph{$(\beta,S)$-absolute winning} game is played as follows: As before, Bob begins by choosing a pair $\omega_1\in \Omega$. After Bob's $n$th move $\omega_n = (x_n,\rho_n)$, Alice chooses a pair $\omega_n' = (x_n',\rho_n')\in\Omega$ satisfying $\rho_n' \leq \beta \rho_n$. She is said to ``delete'' the associated ball $B(\omega_n')$. Bob must then choose a ball $\omega_{n + 1} = (x_{n + 1},\rho_{n + 1}) \leq_s \omega_n$ such that
\begin{equation}
\label{absolute_winning_rules}
\rho_{n + 1} \geq \beta \rho_n \text{ and } B(\omega_{n + 1}) \cap B(\omega_n') = \emptyset.
\end{equation}
Such a choice is always possible due to \eqref{cuniformlyperfect}; cf. \cite[Lemma 4.3]{FSU4}. Bob's choices result in a decreasing sequence of sets
\[
B(\omega_1) \supset  \ldots\supset B(\omega_n) \supset B(\omega_{n + 1}) \supset \ldots
\]
As before, let $I = \bigcap_n B(\omega_n)$, and call Alice the winner if $I\cap S\neq\emptyset$. If Alice has a winning strategy for the $(\beta,S)$-absolute winning game, then $S$ is called \emph{$\beta$-absolute winning}. If $S$ is $\beta$-absolute winning for all $0 < \beta < c/5$, then we say $S$ is \emph{absolute winning}.

\begin{remark}
If $(X,\dist)$ is a Banach space, then the $(\beta,S)$-absolute winning game may be played for any $0 < \beta < 1/3$; the hypothesis $\beta < c/5$ is not necessary. For a proof, see for example the proof of Theorem \ref{maintheorem}(iv).
\end{remark}

\begin{remark}
In the absolute winning game, Alice has rather limited control over the situation, since she can block very few of Bob's possible moves at each step. Thus, being absolute winning is a rather strong property of a set.
\end{remark}

The following proposition summarizes some important properties of winning, strong winning, and absolute winning subsets of a complete uniformly perfect metric space $(X,\dist)$:
\begin{proposition}
\label{properties}
~
\begin{itemize}
\item[(i)] Winning (resp.,  strong winning, absolute winning) sets are dense. 
\item[(ii)] Absolute winning implies strong winning, and strong winning implies winning.
\item[(iii)]  The countable intersection of $\alpha$-winning (resp., $\alpha$-strong winning, absolute winning) sets is again $\alpha$-winning (resp.,  $\alpha$-strong winning, absolute winning).
\end{itemize}
\end{proposition}
\begin{proof}
Indeed, (i) follows directly from the definition. The proof of (ii) is fairly straightforward (cf. \cite[p. 2-3]{McMullen_absolute_winning}), and the basic idea of the proof of (iii) can be found in \cite[Theorem 2]{Schmidt1}.
\end{proof}

\begin{remark}
In \cite[Theorem 7]{Schmidt1}, Schmidt proved that for a large class of games including the above games, the existence of a winning strategy implies the existence of a \emph{positional} winning strategy, i.e. a winning strategy in which a player's moves depend only on his opponent's previous move, and not on the entire history.
\end{remark}

\subsection{Determinacy of games}
\label{subsectionmain}

Inspired by the Banach--Mazur game, infinite games have been studied for almost a century. In 1953, the notion of two-player zero-sum infinite games with perfect information was introduced and systematically studied by D. Gale and F. M. Stewart \cite{GaleStewart}. As stated above, the main purpose of this paper is to study the determinacy of Schmidt's game and its variations. (For set theoretic results of a somewhat different flavor concerning the game on the real line, see \cite{KysiakZoli}). A two-player game is called \emph{determined} if one of the players has a winning strategy.

As a corollary of D. A. Martin's celebrated Borel determinacy theorem \cite{Martin_determinacy}, we deduce the following (see Theorem \ref{theoremdeterminedgeneral}):

\begin{theorem}
\label{theoremdetermined}
Let $(X,\dist)$ be a complete metric space. Fix $0 < \alpha,\beta < 1$ and $S\subset X$. If $S$ is Borel, then Schmidt's $(\alpha,\beta,S)$-game, the $(\alpha,\beta,S)$-strong winning game, and the $(\beta,S)$-absolute winning game are determined. (In the last case, we assume that $X$ is $c$-uniformly perfect with $\beta < c/5$.)
\end{theorem}

We may ask whether these games remain determined when $S$ is not a Borel set. In this paper, we consider the worst case scenario: when $S$ is a \emph{Bernstein set}, meaning that every closed perfect subset of $X$ intersects both $S$ and $X\butnot S$. Bernstein sets are pathological in that they are not measurable with respect to any non-atomic regular measure and do not have the Baire property \cite[Theorem 5.4]{Oxtoby}. Every uncountable Polish space contains a Bernstein subset \cite[Theorem 5.3]{Oxtoby}.\footnote{Both of the above results are stated for $\R$ but easily generalize. In the latter result, the assumption that $X$ is an uncountable Polish space guarantees that the number of closed subsets of $X$ is equal to $\#(X)$.} We may now state our main theorem (see \sectionsymbol\ref{subsectionindeterminacy} for the proof):

\begin{theorem}
\label{maintheorem}
Let $(X,\dist)$ be a complete $c$-uniformly perfect metric space, and let $S\subset X$ be a Bernstein subset. Fix $0 < \alpha,\beta < 1$.
\begin{itemize}
\item[(i)] If $\alpha,\beta < c/(1 + 2c)$, then Schmidt's $(\alpha,\beta,S)$-game and the $(\alpha,\beta,S)$-strong winning game are undetermined.
\item[(ii)] If $(X,\dist)$ is a Banach space and if $1 + \alpha\beta > 2\max(\alpha,\beta)$, then Schmidt's $(\alpha,\beta,S)$-game and the $(\alpha,\beta,S)$-strong winning game are undetermined.
\item[(iii)] If $\beta < (c/5)^2$, then the $(\beta,S)$-absolute winning game is undetermined.
\item[(iv)] If $(X,\dist)$ is a Banach space and if $\beta < 1/3$, then the $(\beta,S)$-absolute winning game is undetermined.
\end{itemize}
\end{theorem}

\begin{remark}
The hypotheses of Theorem \ref{maintheorem} are irrelevant in applications, since one generally cares only about small values of $\alpha$ and $\beta$ (cf. the definitions of $\alpha$-winning and winning sets). Moreover, if $1 + \alpha\beta \leq 2\max(\alpha,\beta)$, then Schmidt's $(\alpha,\beta,S)$-game and the $(\alpha,\beta,S)$-strong winning game are determined for all $S$ by \cite[Lemmas 5 and 6]{Schmidt1}. Thus if $(X,\dist)$ is a Banach space, then the hypotheses of Theorem \ref{maintheorem} cannot be weakened.
\end{remark}

\begin{remark}
One may define \emph{Schmidt's $S$-game} as follows: First Alice chooses $\alpha\in(0,1)$, then Bob chooses $\beta\in(0,1)$, and then Alice and Bob play Schmidt's $(\alpha,\beta,S)$-game. Note that a set $S$ is winning if and only if Alice has a winning strategy for Schmidt's $S$-game. It follows immediately from Theorem \ref{maintheorem} that Schmidt's $S$-game is undetermined whenever $S$ is a Bernstein set. Similar statements may be made about the strong and absolute winning games.
\end{remark}

{\bf Acknowledgements.} The first-named author was supported in part by the Simons Foundation grant \#245708.

\section{A more general context}
\label{sectiongeneral}

To prove Theorems \ref{theoremdetermined} and \ref{maintheorem}, we will introduce a larger class of games which includes the games of \sectionsymbol\ref{subsectiongames}. Before defining this class, we introduce an auxiliary class, the class of \emph{Gale--Stewart} games.

\subsection{Gale--Stewart games}

Let $E$ be a nonempty set, and let $E^*$ and $E^\N$ denote the set of finite and infinite words with alphabet $E$, respectively. Elements of $E^*\cup E^\N$ will be called \emph{plays}. For each play $\omega\in E^*\cup E^\N$, we will denote the length of $\omega$ by $|\omega|$, so that $|\omega| = \infty$ if $\omega\in E^\N$.

Let $\PP^*(E)$ denote the collection of nonempty subsets of $E$, and fix a map $\Gamma:\coprod_{n\in\N}E^{n - 1}\to \PP^*(E)$. The map $\Gamma$ will be called the \emph{ruleset}. A play $\omega\in E^*\cup E^\N$ will be called \emph{legal} if $\omega_n \in \Gamma(\omega_1,\ldots,\omega_{n - 1})$ for all $n\leq |\omega|$. The set of all legal plays will be denoted $E_\Gamma^*\cup E_\Gamma^\N$.

Fix a set $\SS\subset E_\Gamma^\N$ (the \emph{target set}). Alice and Bob play the $(E,\Gamma,\SS)$-\emph{Gale--Stewart game} as follows: Alice and Bob alternate choosing elements $\omega_1,\omega_2,\ldots \in E$, with the odd elements chosen by Bob and the even elements by Alice. The elements they choose must satisfy $(\omega_1,\ldots,\omega_n)\in E_\Gamma^*$ for all $n\in\N$. Such a choice is always possible. If the resulting play $\omega = (\omega_1,\omega_2,\ldots)\in E^\N$ is in $\SS$, then Alice wins; otherwise Bob wins.

An \emph{Alice-strategy} for a Gale--Stewart game $(E,\Gamma,\SS)$ is a map $\sigma_A:\coprod_{n = 0}^\infty E_\Gamma^{2n + 1} \to E$ such that $\sigma_A(\omega)\in \Gamma(\omega)$ for all $\omega\in \coprod_{n = 0}^\infty E_\Gamma^{2n + 1}$. Similarly, a \emph{Bob-strategy} is a map $\sigma_B: \coprod_{n = 0}^\infty E_\Gamma^{2n} \to E$ such that $\sigma_B(\omega)\in \Gamma(\omega)$ for all $\omega\in \coprod_{n = 0}^\infty E_\Gamma^{2n}$. Given any two strategies $\sigma_A$ and $\sigma_B$, there is a unique infinite play $\omega\in E_\Gamma^\N$ such that for each $n\in\N$,
\begin{equation}
\label{omeganformula}
\omega_n = \begin{cases}
\sigma_A(\omega_1,\ldots,\omega_{n - 1}) & n \text{ even}\\
\sigma_B(\omega_1,\ldots,\omega_{n - 1}) & n \text{ odd}
\end{cases}.
\end{equation}
This play represents the game which results if the players Alice and Bob play according to the strategies $\sigma_A$ and $\sigma_B$ respectively. It will be denoted $(\sigma_A,\sigma_B)$. If $(\sigma_A,\sigma_B)\in \SS$ for every Bob-strategy $\sigma_B$, then the Alice-strategy $\sigma_A$ is said to be \emph{winning} for the $(E,\Gamma,\SS)$-Gale--Stewart game. We will equivalently say that $\sigma_A$ \emph{ensures that $\phi\in \SS$}. Here we think of $\phi$ as denoting the outcome of an arbitrary game. Conversely, if $(\sigma_A,\sigma_B)\notin \SS$ for every Alice-strategy $\sigma_A$, then the Bob-strategy $\sigma_B$ is said to be winning for the $(E,\Gamma,\SS)$ game, and we say that the strategy ensures that $\phi\notin \SS$. Clearly, both players cannot possess a winning strategy. The $(E,\Gamma,\SS)$-Gale--Stewart game is said to be \emph{determined} if one of the players possesses a winning strategy.

Alternatively, we may characterize the notion of winning strategies in terms of compatible plays.

\begin{notation}
Given $\omega,\tau\in E^*\cup E^\N$, we write $\omega\preceq\tau$ if $\tau$ is an extension of $\omega$, i.e. if $\omega_n = \tau_n$ for all $n\leq |\omega|$.
\end{notation}

\begin{definition}
\label{definitionsigmacompatible}
Let $\sigma_A$ be an Alice-strategy. A play $\omega\in E_\Gamma^*\cap E_\Gamma^\N$ is called \emph{$\sigma_A$-compatible} if $\omega_{2n} = \sigma_A(\omega_1, \ldots , \omega_{2n - 1})$ whenever $2n \leq |\omega|$. Similarly, given a Bob-strategy $\sigma_B$, a play $\omega  \in E_\Gamma^* \cup E_\Gamma^\N$ is called \emph{$\sigma_B$-compatible} if $\omega_{2n + 1} = \sigma_B(\omega_1, \ldots , \omega_{2n})$ whenever $2n + 1 \leq |\omega|$.

Given a strategy $\sigma$, denote by $E_\sigma^*$ (resp. $E_\sigma^\N$) the set of $\sigma$-compatible plays in of $E_\Gamma^*$ (resp. $E_\Gamma^\N$).
\end{definition}

\begin{remark}
Under this definition, an Alice-strategy $\sigma_A$ is winning if and only if $E_{\sigma_A}^\N\subset \SS$, and a Bob-strategy $\sigma_B$ is winning if and only if $E_{\sigma_B}^\N \cap \SS = \emptyset$.
\end{remark}

One of the main problems in game theory is to classify which Gale--Stewart games are determined. The best general result in this regard is the following:

\begin{theorem}[D. A. Martin \cite{Martin_determinacy}]
\label{theoremboreldeterminacy}
Let $E$ be a nonempty set, and let $\SS\subset E_\Gamma^\N$ be a Borel set, where $E_\Gamma^\N$ is viewed as a subspace of $E^\N$ equipped with the product topology (viewing $E$ as a discrete topological space). Then the $(E,\Gamma,\SS)$-Gale--Stewart game is determined.
\end{theorem}

On the other hand, the following result is well-known (e.g. \cite[p. 137, paragraph 8]{Kechris}):

\begin{proposition}
\label{propositiongalestewartundetermined}
If $\SS$ is a Bernstein set and if $\Gamma(\omega) = E$ for all $\omega\in \coprod_{n\in\N}E^{n - 1}$, then the $(E,\Gamma,\SS)$-Gale--Stewart game is undetermined.
\end{proposition}

\subsection{The games of \sectionsymbol\ref{subsectiongames} as Gale--Stewart games}
\label{subsectionreinterpretation}

Based on Theorem \ref{theoremboreldeterminacy} and Proposition \ref{propositiongalestewartundetermined}, one might guess that the games described in \sectionsymbol\ref{subsectiongames} exhibit the same behavior, i.e. that they are determined on Borel sets but undetermined on Bernstein sets. And indeed, Theorems \ref{theoremdetermined} and \ref{maintheorem} shows that this is essentially correct. To prove these theorems, we will relate the games of \sectionsymbol\ref{subsectiongames} to Gale--Stewart games.

Let $(X,\dist)$ be a complete metric space. Fix $\alpha,\beta\in(0,1)$ and $S\subset X$, and consider Schmidt's $(\alpha,\beta,S)$-game. We will consider a corresponding Gale--Stewart game which is determined if and only if Schmidt's $(\alpha,\beta,S)$-game is determined. Let $E = \Omega$ be the set of formal balls in $X$. Define the ruleset $\Gamma = \Gamma_{\alpha,\beta}:\coprod_{n\in\N} E^n\to E$ according to the rules of Schmidt's $(\alpha,\beta)$-game; namely, $E\ni \omega_n\in \Gamma(\omega_1,\ldots,\omega_{n - 1})$ if and only if
\[
\omega_n \leq_s \omega_{n - 1} \text{ and } \rho_n = \begin{cases}
\alpha \rho_{n - 1} & n \text{ even}
\\ \beta \rho_{n - 1} & n \text{ odd}
\end{cases}.
\]
Let
\[
\SS_{\alpha,\beta} = \left\{\omega\in E_\Gamma^\N : \bigcap_{n\in\N}B(\omega_{2n + 1}) \cap S\neq\emptyset\right\}.
\]
Then the $(E,\Gamma_{\alpha,\beta},\SS_{\alpha,\beta})$-Gale--Stewart game is equivalent to Schmidt's $(\alpha,\beta,S)$-game, in the sense that every winning Alice-strategy (resp. Bob-strategy) for the $(E,\Gamma_{\alpha,\beta},\SS_{\alpha,\beta})$-Gale--Stewart game corresponds to a winning Alice-strategy (resp. Bob-strategy) for Schmidt's $(\alpha,\beta,\SS)$-game, and vice-versa. In particular, each game is determined if and only if the other is.

Similarly, for the $(\alpha,\beta,S)$-strong winning game and the $(\beta,S)$-absolute winning game, there exist rulesets $\Gamma_{\alpha,\beta}^{\text{strong}}$ and $\Gamma_{\beta}^{\text{absolute}}$ and sets $\SS_{\alpha,\beta}^{\text{strong}}$ and $\SS_{\beta}^{\text{absolute}}$ such that the $(E,\Gamma_{\alpha,\beta}^{\text{strong}},\SS_{\alpha,\beta}^{\text{strong}})$-Gale--Stewart game is equivalent to the $(\alpha,\beta,S)$-strong winning game and the $(E,\Gamma_\beta^{\text{absolute}},\SS_{\beta}^{\text{absolute}})$-Gale--Stewart game is equivalent to the $(\beta,S)$-absolute winning game.

\subsection{Games played on complete metric spaces}
We can generalize the ideas of \sectionsymbol\ref{subsectionreinterpretation} to the setting of \emph{games played on complete metric spaces}. Such games will depend on a nonempty set $E$, a ruleset $\Gamma:\coprod_{n\in\N}E^{n - 1}\to \PP^*(E)$, a complete metric space $(X,\dist)$, and on two additional parameters $(I,S)$.

The parameter $I$ is the \emph{topological interpretation} of the game, and it is a map $I:E_\Gamma^*\to\KK^*(X)$, where $\KK^*(X)$ denotes the set of all nonempty closed subsets of $X$, which is order-preserving in the sense that $\omega\preceq \tau$ implies $I(\tau)\subset I(\omega)$. Given a topological interpretation $I:E_\Gamma^*\to\KK^*(X)$, we can define $I:E_\Gamma^\N\to \KK^*(X)$ by the formula $I(\omega) = \bigcap_{n = 0}^\infty I(\omega_1,\ldots,\omega_n)$.

The parameter $S$ is the \emph{target set}, and it is a subset of $X$. The $(E,\Gamma,I,S)$-game is played as follows: Players Alice and Bob take turns playing legal moves, resulting in a play $\omega\in E_\Gamma^\N$ which defines the \emph{outcome} $I(\omega)\in\KK^*(X)$. If $I(\omega)\cap S\neq\emptyset$, then Alice wins; otherwise Bob wins. We say that the $(E,\Gamma,I,S)$-game is a \emph{game played on the metric space $(X,\dist)$}.

Based on the above description, it is easy to see that each $(E,\Gamma,I,S)$-game is equivalent to some Gale--Stewart game, namely the Gale--Stewart game $(E,\Gamma,\SS_{I,S})$ where
\begin{equation}
\label{Sprimedef}
\SS_{I,S} = \{\omega\in E_\Gamma^\N :I(\omega)\cap S\neq\emptyset\}.
\end{equation}
Moreover, the games of \sectionsymbol\ref{subsectiongames} are all examples of games played on complete metric spaces, with $I(\omega) = \bigcap_{n < |\omega|/2} B(\omega_{2n + 1})$.

We mention in passing that the well-known Banach--Mazur game is also an example of a game played on a complete metric space. It satisfies the hypotheses of the two main theorems in the next section, and is therefore determined on Borel sets and undetermined on Bernstein sets. However, both these results are already known as a consequence of the fact that Alice has a winning strategy for the Banach--Mazur game if and only if the target set is comeager; see \cite[p. 152, last paragraph/p. 153, first paragraph]{Kechris} and \cite[p. 29, last paragraph]{Oxtoby}, respectively. In fact, the former reference shows that the Banach--Mazur game is determined on analytic sets. The corresponding question of whether the games of \sectionsymbol\ref{subsectiongames} are determined on analytic sets remains open.

\section{Determinacy and indeterminacy of games played on complete metric spaces}
In this section, we state and prove generalizations of the theorems of \sectionsymbol\ref{subsectionmain}.

\subsection{Determinacy on Borel sets}
We will prove the following generalization of Theorem \ref{theoremdetermined}:

\begin{theorem}
\label{theoremdeterminedgeneral}
Let $(E,\Gamma,I,S)$ be a game played on a complete metric space $(X,\dist)$, and suppose that $S\subset X$ is Borel. Suppose in addition that
\begin{equation}
\label{conditiondense}
\text{For all }\omega\in E_\Gamma^\N, \text{ either } \diam(I(\omega_1,\ldots,\omega_n)) \tendsto n 0 \text{ or } I(\omega)\cap S\neq\emptyset.
\end{equation}
Then the $(E,\Gamma,I,S)$-game is determined.
\end{theorem}

\begin{remark}
For the games of \sectionsymbol\ref{subsectiongames}, the condition \eqref{conditiondense} holds as long as $S$ is dense.
\end{remark}
\begin{proof}
Fix $\omega\in E_\Gamma^\N$, and for each $n$ let $\omega_n = (x_n,\rho_n)$. By assumption we have $(x_{2m + 1},\rho_{2m + 1}) \leq_s (x_{2n + 1},\rho_{2n + 1})$ for all $m\geq n$, so
\begin{equation}
\label{formalinclusion}
\dist(x_{2m + 1},x_{2n + 1}) \leq \rho_{2n + 1} - \rho_{2m + 1}.
\end{equation}
Since the sequence $(\rho_{2n + 1})_0^\infty$ is decreasing and bounded, it is Cauchy and therefore the right hand side of \eqref{formalinclusion} tends to zero as $m,n\to\infty$. Thus the sequence $(x_{2n + 1})_0^\infty$ is Cauchy. Since $X$ is complete, we have $x_{2n + 1}\to x$ for some $x\in X$. Then $I(\omega)\supset B(x,\rho)$ where $\rho = \lim_{n\to\infty}\rho_n$. Since $S$ is dense, we have $I(\omega)\cap S\neq\emptyset$ if $\rho > 0$.
\end{proof}

On the other hand, if $S$ is not dense, then Bob can win on the first turn. Thus Theorem \ref{theoremdeterminedgeneral} implies Theorem \ref{theoremdetermined}.

\begin{proof}[Proof of Theorem \ref{theoremdeterminedgeneral}]
By Theorem \ref{theoremboreldeterminacy}, it suffices to show that the set $\SS_{I,S}$ defined in \eqref{Sprimedef} is Borel. Let
\begin{equation}
\label{Zdef}
Z = \{\omega\in E_\Gamma^\N:\diam(I(\omega_1,\ldots,\omega_n)) \tendsto n 0\}.
\end{equation}
Clearly, $Z$ is Borel. On the other hand, $E_\Gamma^\N\butnot Z\subset \SS_{I,S}$ by \eqref{conditiondense}. So it suffices to show that $\SS_{I,S}\cap Z$ is Borel. Since $X$ is complete, we have $\#(I(\omega)) = 1$ for all $\omega\in Z$. Thus we may define a map $\iota:Z\to X$ by letting $\{\iota(\tau)\} = I(\tau)$. We observe that $\SS_{I,S}\cap Z = \iota^{-1}(S)$. Thus to complete the proof it suffices to show the following:
\begin{lemma}
\label{lemmacontinuousZ}
The map $\iota$ is continuous.
\end{lemma}
\begin{subproof}
Suppose that $U$ is an open subset of $X$, and fix $x\in U$ and $\omega\in \iota^{-1}(x)$. Since $\omega\in Z$, we have $\diam(I(\omega_0,\ldots,\omega_n)) < \dist(x,X\butnot U)$ for some $n\in\N$. Then the set $V = \{\tau \in E_\Gamma^\N: \tau_i = \omega_i, 0 \leq i \leq n\}$ satisfies $\omega\in V\cap Z \subset \iota^{-1}(U)$; moreover, $V$ is open in the topology of $E_\Gamma^\N$. Since $x$ and $\omega$ were arbitrary, $\iota$ is continuous.
\end{subproof}
\end{proof}

\ignore{
\begin{notation}
If $X$ is a topological space, let $\KK^*(X)$ denote the set of all closed subsets of $X$.
\end{notation}

\begin{definition}
Let $E$ be a nonempty set. A \emph{topological embedding} of the tree $E^*$ into a complete metric space $(X,\dist)$ is a map $I:E^*\to \KK^*(X)$ such that $\omega\preceq \tau$ implies $I(\tau)\subset I(\omega)$. Given a topological embedding $I:E^*\to\KK^*(X)$, we can define $I:E^\N\to \KK^*(X)$ by the formula $I(\omega) = \bigcap_{n = 0}^\infty I(\omega_1,\ldots,\omega_n)$.

If $I:E^*\to\KK^*(X)$ is a topological embedding and if $S\subset \KK^*(X)$, we can consider the Gale--Stewart game $(E,I^{-1}(S))$. Moreover, if $S\subset X$, we can let
\[
\SS_{I,S} \df \{A\in\PP(X): A\cap S\neq\emptyset\},
\]
and consider the Gale--Stewart game $(E,I^{-1}(\SS_{I,S}))$.
\end{definition}

\begin{example}[Banach--Mazur Game]
The most well-known instance of a Gale--Stewart game is the Banach--Mazur game. Given a complete metric space $(X,\dist)$, the Banach--Mazur game on $X$ can be described as follows. In the first move of the game, Bob chooses an ordered pair $(x,\rho _0)\in (X\times \R _{>0})$, thus specifying a closed ball of positive radius centered on $x$. In what follows, Alice and Bob alternatively choose ordered pairs in the same manner specifying balls of positive radius under the constraint that each one is contained in the previous ball of choice of the other player. In term of the Gale--Stewart model the Banach--Mazur game can be described as follows:
\begin{enumerate}
\item $A_n = X \times \R_{>0} = \{ (x, \rho): x \in X, \rho > 0 \}$ for
$n > 0$.
\item $\Gamma(\omega_0) = A_1$.
\item $\Gamma(\omega_0, (x_1, \rho_1), \ldots, (x_n, \rho_n)) = \{ (x, \rho) \in
A_{n + 1}: d(x_n, x) + \rho \leq \rho_n \}$ for $n > 0$.
\end{enumerate}
We then let $I:E^*\to \KK^*(X)$ be the map $I(\omega) = B(\omega_{|\omega|})$, where $B: \coprod_{n > 0} A_n \longrightarrow \KK^*(X)$ is the function $(x, \rho) \mapsto B(x, \rho) \df \{ y \in X: d(x, y) \leq \rho \}$.

For the Banach--Mazur game, Alice has a winning strategy if and only if $S$ is comeager, and Bob has a winning strategy if and only if $S\cap B$ is meager for some ball $B\subset X$. See~\cite{Oxtoby} for more details. \\
\end{example}

\begin{example}[Schmidt's game]
Another instance of a Gale--Stewart game is Schmidt's $(\alpha, \beta)$-game which was described in the introduction. In terms of the Gale--Stewart setup, the scheme for Schmidt's game is given by:
\begin{enumerate}
\item $A_n = X \times \R_{>0} = \{ (x, \rho): x \in X, \rho > 0 \}$ for all $n > 0$.
\item $\Gamma(\omega_0) = A_1$.
\item $\Gamma(\omega_0, (x_1, \rho_1), \ldots, (x_{2n + 1}, \rho_{2n + 1})) = \{(x, \rho) \in A_{2n + 2}: d(x_{2n + 1}, x) + \alpha\rho \leq \rho_{2n + 1} \}$ for $n \geq 0$.
\item $\Gamma(\omega_0, (x_1, \rho_1), \ldots, (x_{2n}, \rho_{2n})) = \{ (x,
\rho) \in A_{2n + 1}: d(x_{2n}, x) + \beta\rho \leq \rho_{2n} \}$ for $n > 0$.
\end{enumerate}
As mentioned before, since $\rho_{2n + 1} = (\alpha\beta)^n\rho_1$ and $\rho_{2n} = \beta(\alpha\beta)^{n - 1}\rho_1$, we have $\lim \rho_n = 0$, and so each play in $E^\N$ corresponds to a single point  $x_{\infty}\in X$.
\end{example}

\subsection{Determinacy for topological embeddings}

\begin{proposition}
\label{propositiondeterminacy}
Let $I:E^*\to\KK^*(X)$ be a topological embedding of a Gale--Stewart game $(E,\Gamma)$, and suppose that every nonempty open subset of $X$ is a superset of some element of $\Gamma(\omega_0)$. Then every Borel subset of $X$ is determined.
\end{proposition}
The condition on $\Gamma(\omega_0)$ is satisfied for the Banach--Mazur game and for all variants of Schmidt's game.
\begin{proof}
We first prove the following lemma:
\begin{lemma}
\label{lemmacontinuousZ}
The map $\iota:Z\to X$ defined by $\{\iota(\tau)\} = I(\tau)$ is continuous.
\end{lemma}
\begin{proof}
Suppose that $U$ is an open subset of $X$, and fix $x\in U$ and $a\in \iota^{-1}(x)$. Since $a\in Z$, we have $\diam(I(\omega_0,\ldots,\omega_n)) < \dist(x,X\butnot U)$ for some $n\geq 0$. Then the set $V = \{b \in E^\N: b_i = \omega_i, 0 \leq i \leq n\}$ satisfies $a\in V\cap Z \subset \iota^{-1}(U)$. Since $V$ is open in the topology of $E^\N$, so is $\iota^{-1}(U)$, and thus $\iota$ is continuous.
\end{proof}
Now assume $S\subset X$ is Borel. Then $\iota^{-1}(S)$ is a Borel subset of $Z$, which is in turn a Borel subset of $E^\N$. Now if $S$ is not dense, then Bob can win on the first turn by the assumption on $\Gamma(\omega_0)$, so let us assume that $S$ is dense. Together with properness, this implies that Alice wins if the radii do not tend to zero. Thus
\[
\SS_{I,S} = \iota^{-1}(S)\cup (E^\N\butnot Z)
\]
is a Borel set. By Martin's theorem, the game is determined on this set, and the conclusion follows.
\end{proof}
}% end ignore

\subsection{Indeterminacy on Bernstein sets}
\label{subsectionindeterminacy}
To state the appropriate generalization of Theorem \ref{maintheorem} to the setting of games played on complete metric spaces, we need to introduce some more terminology.

\begin{definition}
Let $(E,\Gamma,\SS)$ be a Gale--Stewart game. Given $\tau\in E_\Gamma^*$, an Alice-strategy $\sigma_A$, and a Bob-strategy $\sigma_B$, there is a unique infinite play $\omega\succeq\tau$ satisfying \eqref{omeganformula} for all $n > |\tau|$. This play represents the game which results if the players Alice and Bob first make the moves $\tau_1,\ldots,\tau_{|\tau|}$, and then play according to the strategies $\sigma_A$ and $\sigma_B$ respectively. It will be denoted $(\tau,\sigma_A,\sigma_B)$. If $(\tau,\sigma_A,\sigma_B)\in \SS$ for every Bob-strategy $\sigma_B$, then the Alice-strategy $\sigma_A$ is said to be \emph{winning relative to $\tau$} for the $(E,\Gamma,\SS)$-Gale--Stewart game. We say that the strategy $\sigma_A$ \emph{ensures that $\phi\in \SS$ assuming $\phi\succeq\tau$}. Conversely, if $(\sigma_A,\sigma_B)\notin \SS$ for every Alice-strategy $\sigma_A$, then the Bob-strategy $\sigma_B$ is said to be winning relative to $\tau$ for the $(E,\Gamma,\SS)$-Gale--Stewart game, and we say that $\sigma_B$ ensures that $\phi\notin \SS$ assuming $\phi\succeq\tau$.
\end{definition}

\begin{definition}
Let $(E,\Gamma,I,S)$ be a game played on a complete metric space $(X,\dist)$. Let $Z$ be defined as in \eqref{Zdef}, and for each $x\in X$ let
\[
N_x = \{\omega\in E_\Gamma^\N: x\notin I(\omega)\}.
\]
\begin{itemize}
\item The $(E,\Gamma,I,S)$-game is \emph{shrinking} if for every $\omega\in E_\Gamma^*$, both players have strategies ensuring $\phi\in Z$ assuming $\phi\succeq\omega$.
\item The $(E,\Gamma,I,S)$-game is \emph{nondegenerate} if for every $\omega\in E_\Gamma^*$ and $x\in X$, both players have strategies ensuring $\phi\in N_x$ assuming $\phi\succeq\omega$.
\end{itemize}
In other words, the game is shrinking if each player can ensure that the outcome is in $Z$, even if for finitely many moves that player does not have free will. Similarly, the game is nondegenerate if each player can ensure that the outcome is in $N_x$, under the same stipulations. More informally, the game is shrinking if each player can force the diameters to tend to zero, and is nondegenerate if each player can avoid an arbitrary specified point. (However, cf. Remark \ref{remarkshrinking}.)
\end{definition}

\begin{theorem}
\label{theoremindeterminacy}
Let $(E,\Gamma,I,S)$ be a shrinking nondegenerate game played on a complete metric space $(X,\dist)$. If Alice (resp. Bob) has a winning strategy for the $(E,\Gamma,I,S)$-game, then $S$ (resp. $X\butnot S$) contains a closed perfect set. In particular, if $S$ is a Bernstein set then the game is undetermined.
\end{theorem}
\begin{proof}
In fact we will demonstrate the following slightly weaker assertion: If Alice (resp. Bob) has a strategy ensuring that $\phi\in Z$ implies $\iota(\phi)\in S$ (resp. that $\phi\in Z$ implies $\iota(\phi)\notin S$), then $S$ (resp. $X\butnot S$) contains a closed perfect set. The advantage of proving the weaker assertion is that it is symmetric with respect to interchanging Alice and Bob, so without loss of generality we suppose that Alice has a strategy $\sigma_A$ which ensures that $\phi\in Z$ implies $\iota(\phi)\in S$. Recall that $E_{\sigma_A}^*$ denotes the set of finite $\sigma_A$-compatible plays (cf. Definition \ref{definitionsigmacompatible}).

\begin{lemma}
\label{lemmasplit}
For every $\omega\in E_{\sigma_A}^*$ and $r > 0$, there exist $\tau_1,\tau_2\succeq \omega$ in $E_{\sigma_A}^*$ such that $\diam(I(\tau_i)) < r$ and $I(\tau_1)\cap I(\tau_2) = \emptyset$.
\end{lemma}
\begin{subproof}
Since $(E,\Gamma,I,S)$ is shrinking, Bob has a strategy $\sigma_B$ ensuring that $\phi\in Z$ assuming $\phi\succeq\omega$; it follows that the play $\tau_1^\infty \df (\omega,\sigma_A,\sigma_B)$ is $\sigma_A$-compatible and satisfies $\omega\preceq \tau_1^\infty\in Z$. Since $\tau_1^\infty\in Z$ and $X$ is complete, $I(\tau_1^\infty)$ is a singleton, say $I(\tau_1^\infty) = \{x\}$. Since $(E,\Gamma,I,S)$ is nondegenerate, Bob has a strategy $\w\sigma_B$ ensuring that $\phi\in N_x$ assuming $\phi\succeq\omega$; it follows that the play $\tau_3^\infty \df (\omega,\sigma_A,\w\sigma_B)$ is $\sigma_A$-compatible and satisfies $\omega\preceq \tau_3^\infty\in N_x$. Since $x\notin I(\tau_3^\infty) = \bigcap_{\tau_3\preceq \tau_3^\infty} I(\tau_3)$, there exists a finite play $\omega\preceq \tau_3\preceq \tau_3^\infty$ satisfying $x\notin I(\tau_3)$. Again using the fact that $(E,\Gamma,I,S)$ is shrinking, there is a $\sigma_A$-compatible play $\tau_2^\infty\in Z$ satisfying $\tau_3\preceq \tau_2^\infty$. Again $I(\tau_2^\infty)$ is a singleton, say $I(\tau_2^\infty) = \{y\}$. We have $x\neq y$. Let $\rho = \min(\frac{1}{2}\dist(x,y),r) > 0$, and let $\tau_1\preceq \tau_1^\infty$ and $\tau_2\preceq \tau_2^\infty$ satisfy $\diam(I(\tau_i)) < \rho$. Since $I(\tau_i^\infty)\subset I(\tau_i)$, this completes the proof.
\end{subproof}

Let $B^\N$ be the set of all infinite binary sequences, and $B^*$ be the set of all finite binary sequences. We will denote the concatenation operator on $B^*$ by the symbol $*$. We inductively define a map $g:B^*\to E_{\sigma_A}^*$ sending the empty string of $B^*$ to the empty string of $E^*$ via the following rule:
\begin{quote}
If $\omega = g(\theta)$ for some $\theta\in B^*$, choose $\sigma_A$-compatible plays $\tau_1,\tau_2\succeq \omega$ satisfying $\diam(I(\tau_i)) < 2^{-|\theta|}$ and $I(\tau_1)\cap I(\tau_2) = \emptyset$, as guaranteed by Lemma \ref{lemmasplit}. Let $g(\theta * i) \df \tau_i$ for $i = 1,2$.\\
\end{quote}
Let $\overline g: B^\N \rightarrow E_{\sigma_A}^\N$ denote the natural extension of $g$, and observe that that $\overline g$ is continuous. Moreover, $\overline g: B^\N\to Z$. By Lemma \ref{lemmacontinuousZ}, $\iota\circ\overline g:B^\N\to X$ is continuous, and by construction, $\iota\circ\overline g$ is injective; thus $\iota\circ\overline g$ is a homeomorphic embedding. On the other hand, since $\sigma_A$ ensures that $\phi\in Z$ implies $\iota(\phi)\in S$, we have $\iota\circ\overline g(B^\N)\subset S$. Thus $S$ contains the closed perfect set $\iota\circ\overline g(B^\N)$.
\end{proof}

Having completed the proof of Theorem \ref{theoremindeterminacy}, we proceed to use it to deduce Theorem \ref{maintheorem}. This consists of verifying that each of the games in \sectionsymbol\ref{subsectiongames} is shrinking and nondegenerate under the hypotheses of Theorem \ref{maintheorem}.

\begin{proof}[Proof of Theorem \ref{maintheorem} using Theorem \ref{theoremindeterminacy}]~

\begin{itemize}
\item[(i)] Schmidt's $(\alpha,\beta,S)$-game and the $(\alpha,\beta,S)$-strong winning game are both shrinking for every complete metric space $(X,\dist)$. Indeed, either player may choose the smallest possible radius on each turn, guaranteeing that the diameters of the balls tend to zero and thus that $\phi\in Z$.

Now suppose that $(X,\dist)$ is $c$-uniformly perfect, and that $\alpha,\beta < c/(1 + 2c)$. We claim that both Schmidt's $(\alpha,\beta,S)$-game and the $(\alpha,\beta,S)$-strong winning game are nondegenerate. Indeed, suppose that one of the players has just made a move $B(x,\rho)$, and suppose the other player, without loss of generality Alice, wants to avoid the point $y\in X$. If $y\notin B(x,\alpha \rho)$, then Alice may choose the ball $B(x,\alpha \rho)$ and thus avoid the point $y$. So suppose $y\in B(x,\alpha \rho)$. Since $X$ is $c$-uniformly perfect, we have $B(y,(1 - 2\alpha)\rho)\butnot B(y,c(1 - 2\alpha)\rho)\neq\emptyset$; fix $z\in B(y,(1 - 2\alpha)\rho)\butnot B(y,c(1 - 2\alpha)\rho)$. Then $B(z,\alpha \rho)\subset B(x,\rho)$ since $\dist(z,x) \leq (1 - \alpha)\rho$. So $B(z,\alpha \rho)$ is a legal move for Alice, and it avoids the point $y$ because $\dist(z,y)\geq c(1 - 2\alpha)\rho > \alpha \rho$ by the condition on $\alpha$.

\item[(ii)] If $(X,\dist)$ is a Banach space, then Schmidt's $(\alpha,\beta,S)$-game is nondegenerate as long as $2\max(\alpha,\beta) < 1 + \alpha\beta$; this follows directly from \cite[Lemma 7]{Schmidt1}.

\item[(iii)] For any $0 < \beta < c/5$, the $\beta$-absolute winning game is shrinking on any complete $c$-uniformly perfect metric space. Indeed, Bob can force the diameters to go to zero by always choosing the smallest possible radius for his ball. Next we will describe Alice's strategy to foce the diameters to go to zero. Suppose Bob has just made his $n$th move $B(x_n,\rho_n)$. Alice's strategy will be to remove the center $\{x_n\}$ (or a small neighborhood thereof). This forces Bob's $(n + 1)$ ball $\omega_{n + 1} = (x_{n + 1},\rho_{n + 1}) \leq_s \omega_n =  (x_n,\rho_n)$ to satisfy $x_n\notin B(x_{n + 1},\rho_{n + 1})$, i.e. $\rho_{n + 1} < \dist(x_n,x_{n + 1})$. On the other hand, by the definition of formal inclusion we have $\dist(x_n,x_{n + 1}) \leq \rho_n - \rho_{n + 1}$. Combining these inequalities gives $\rho_{n + 1} < \rho_n/2$, which implies that $\rho_n\to 0$.

Now suppose that $\beta < (c/5)^2$. We claim that the $(\beta,S)$-absolute winning game is nondegenerate.  Indeed, Alice can avoid a point $x\in X$ by simply deleting a neighborhood of it. Now suppose that Alice has just deleted the ball $B(x',\rho')$, following Bob's move $B(x,\rho)$. To see that Bob can avoid the point $y$, note that by \cite[Lemma 4.3]{FSU4}, there is a ball $B(z,(c/5)\rho)\subset B(x,\rho)\butnot B(x',\rho')$; applying \cite[Lemma 4.3]{FSU4} again gives a ball $B(w,(c/5)^2\rho)\subset B(z,(c/5)\rho)\butnot\{y\}$. By making this move, Bob avoids the point $y$.

\item[(iv)] If $(X,\dist)$ is a Banach space, then for $0 < \beta < 1/3$, the argument of (iii) shows that the $\beta$-absolute winning game is shrinking. We must show that it is also nondegenerate. Indeed, Alice can avoid a point by simply deleting a neighborhood of it. We will now describe Bob's strategy to avoid a point $\xx_0\in X$. The strategy will conist of two phases. We will describe the second phase first. Fix a unit vector $\vv\in X$. Suppose that Bob has just made the move $B(\xx_n,\rho_n)$, and that Alice has deleted the ball $B(\xx_n',\rho_n')$ with $\rho_n' \leq \beta \rho_n$. Consider the balls
\begin{equation}
\label{Bobsoptions}
B(\xx_n + (1 - \beta)\rho_n\vv,\beta \rho_n) \text{ and } B(\xx_n - (1 - \beta)\rho_n\vv,\beta \rho_n).
\end{equation}
It is readily computed that the distance between these balls is precisely $2(1 - 2\beta)\rho$. In particular, since $\beta < 1/3$, this distance is strictly greater than $2\beta\rho$, which is in turn greater than the diameter of Alice's ball $B(\xx_n',\rho_n')$. Thus Alice's ball $B(\xx_n',\rho_n')$ intersects at most one of the balls \eqref{Bobsoptions}. Bob's strategy will be to choose whichever ball Alice's ball does not intersect, i.e. whichever one is legal.

The result of Bob's strategy will be that the outcome of the game takes the form
\begin{equation}
\label{outcomeform}
\xx_n + \sum_{m = 0}^\infty \epsilon_m (1 - \beta)\beta^m \rho_n\vv,
\end{equation}
where $\epsilon_m = \pm 1$ for all $m\in\N$. So during the first phase of his strategy, Bob's goal will be to choose a ball $B(\xx_n,\rho_n)$ such that $\xx_0$ cannot be written in the form \eqref{outcomeform}, or equivalently that
\[
\xx_n \notin S_{\rho_n} :=  \left\{\xx_0 + \sum_{m = 0}^\infty \epsilon_m (1 - \beta)\beta^m \rho_n\vv : \epsilon_m \in \{-1,+1\}\right\}.
\]
We remark that if $\dim(X) \geq 2$, this is easy; Bob may simply choose the ball $B(\xx_n,\rho_n)$ so that $\xx_n$ does not lie on the line $\xx_0 + \R\vv$. This is possible since the set of legal centers for Bob's balls always contains a nonempty open set; see below for details.

On the other hand, suppose that $\dim(X) = 1$, i.e. $X = \R$. Then for each $\rho > 0$, the set $S_\rho$ has Hausdorff dimension $\log_\beta(1/2) < 1$ (cf. \cite{Hutchinson}) and thus Lebesgue measure zero.\footnote{Alternatively, we may compute directly that $\lambda(S_\rho) = 0$:
\begin{align*}
\lambda(S_\rho) &\leq \sum_{\epsilon_0 = \pm 1} \cdots \sum_{\epsilon_{M - 1} = \pm 1} \lambda\left\{\xx_0 + \sum_{m = 0}^\infty \epsilon_m (1 - \beta)\beta^m \rho_n\vv : \epsilon_m \in \{-1,+1\} \all m\geq M\right\}\\
&= 2^M \lambda\left\{\sum_{m = M}^\infty \epsilon_m (1 - \beta)\beta^m \rho_n\vv : \epsilon_m \in \{-1,+1\} \all m\geq M\right\}\\
&\leq 2^M \diam\left\{\sum_{m = M}^\infty \epsilon_m (1 - \beta)\beta^m \rho_n\vv : \epsilon_m \in \{-1,+1\} \all m\geq M\right\} \\
&= 2^{M + 1} \sum_{m = M}^\infty (1 - \beta)\beta^m \rho_n
= (2\beta)^M 2\rho_n \tendsto M 0.
\end{align*}} In particular, $S_\rho$ has empty interior. Now suppose that Bob has just made the move $B(\xx_{n - 1},\rho_{n - 1})$ (the last move where he ``does not have free will''), and that Alice has just deleted the set $B(\xx_{n - 1}',\rho_{n - 1}')$. Then the set of $\xx$ for which Bob can legally play the ball $B(\xx,\beta\rho_{n - 1})$ is the set
\begin{equation}
\label{Bobscenters}
B(\xx_{n - 1},(1 - \beta)\rho_{n - 1}) \butnot B(\xx_{n - 1}',\beta\rho_{n - 1} + \rho_{n - 1}').
\end{equation}
Since $\beta < 1/3$ and $\rho_{n - 1}' \leq \beta \rho_{n - 1}$, we have $(1 - \beta)\rho_{n - 1} > \beta\rho_{n - 1} + \rho_{n - 1}'$. It follows that the set \eqref{Bobscenters} contains a nonempty open set, so Bob can choose his center not to lie in the set $S_{\beta\rho_{n - 1}}$. In this way he avoids the point $\xx_0$.
\end{itemize}
\end{proof}

\begin{remark}
\label{remarkshrinking}
It is \emph{not} true that for a shrinking game, if a player has a strategy ensuring $\phi\in \SS$ then he or she also has a strategy ensuring $\phi\in \SS\cap Z$. For example, let $X = B^\N$ with the metric
\[
\dist(\theta,\psi) = 4^{-\min\{k : \theta_k \neq \psi_k\}}
\]
and consider the set
\[
S = \{\theta\in X: \theta_k = 1 \text{ for infinitely many $n\in\N$}\}.
\]
As we have seen above, the $(\frac 12,\frac 12,S)$-strong winning game is shrinking. However, Alice has a strategy ensuring $\phi\in \iota^{-1}(S)\cup (E_\Gamma^\N\butnot Z)$, but she does not have a strategy ensuring $\phi\in \iota^{-1}(S)$. Indeed, the only turns on which the choices of players affect the outcome of the game are those turns $n$ for which
\begin{equation}
\label{goodturns}
\rho_n < 4^{-k} \leq \rho_{n - 1} \text{ for some } k.
\end{equation}
Either player can ensure that he or she controls all such turns by simply choosing $\rho_n = \rho_{n - 1}$ unless it is possible to choose $\rho_n$ satisfying \eqref{goodturns}, in which case he or she does that instead. This ensures that either the outcome is not in $Z$, or it is precisely what the player wanted, except for the finitely many values which were determined by Bob's initial ball. In particular, Alice has a strategy ensuring that $\phi\in \iota^{-1}(S)\cup (E_\Gamma^\N\butnot Z)$, and Bob has a strategy ensuring that $\phi\in \iota^{-1}(X\butnot S)\cup (E_\Gamma^\N\butnot Z)$. The existence of the latter strategy guarantees that Alice cannot have a strategy ensuring $\phi\in \iota^{-1}(S)$.
\end{remark}

\bibliographystyle{amsplain}

\bibliography{bibliography}

\end{document}